\documentclass[11pt, a4paper,reqno]{amsart}

\usepackage{amssymb}

\usepackage{enumitem}
\setlist[enumerate,1]{label=\rm(\arabic*)}
\setlist[enumerate,2]{label=\rm(\alph*)}
\setlist[enumerate,3]{label=\rm(\roman*)}

\usepackage[PS]{diagrams}
\usepackage{xcolor}
\usepackage{pgf}

\newcommand{\iso}{\cong}
\newcommand{\ie}{\textit{i.e.}}

\newcommand{\ignore}[1]{\relax}
\newcommand{\bN}{\mathbb{N}}
\newcommand{\bZ}{\mathbb{Z}}

\newcommand{\bC}{\mathbb{C}}
\newcommand{\nbd}{\nobreakdash}

\newcommand{\tensor}{\otimes}
\newcommand{\ann}{\ensuremath{\mathrm{ann}}}

\numberwithin{equation}{section}

\newtheorem{theorem}[equation]{Theorem}
\newtheorem{corollary}[equation]{Corollary}
\newtheorem{proposition}[equation]{Proposition}
\newtheorem{lemma}[equation]{Lemma}

\theoremstyle{definition}
\newtheorem{definition}[equation]{Definition}
\newtheorem{remark}[equation]{Remark}

\newcommand{\cF}{\ensuremath{\mathcal{F}}}
\newcommand{\cK}{\ensuremath{\mathcal{K}}}

\let\ideal=\unlhd

\usepackage[hypertexnames=false,colorlinks=false,pdfborderstyle={/S/U/W 1}]{hyperref}
\usepackage{bookmark}

\newcommand{\M}{\mathcal{M}}
\newcommand{\supp}{\mathrm{supp}}
\newcommand{\rank}{\mathrm{rank}}
\newcommand{\mcrank}{\mathrm{rank}_{\text{\textsc{McCoy}}}}
\newcommand{\inv}{^{-1}}

\begin{document}

\title[Algebraic jump loci over Laurent polynomial rings]{Algebraic jump loci
  for rank and Betti numbers over Laurent polynomial rings}

\date{November 27, 2018}

\author{Thomas H\"uttemann}

\address{Thomas H\"uttemann\\ Mathematical Sciences Research Centre\\
  School of
  Mathematics and Physics\\ Queen's University Belfast\\
  Belfast BT7~1NN\\ Northern Ireland, UK}

\email{t.huettemann@qub.ac.uk}

\author{Zuhong Zhang}

\address{Zhang Zuhong\\ School of Mathematics\\ Beijing Institute Of
  Technology\\ \break 5 South Zhongguancun Street, Haidian District\\
  100081 Beijing\\ P.~R.~China}

\email{zuhong@gmail.com}

\thanks{Work on this paper commenced during a research visit of the
  first author to Beijing Institute of Technology in
  January~2017. Their hospitality and financial support are gratefully
  acknowledged.}


\begin{abstract}
  Let $C$ be a chain complex of finitely generated free modules over a
  commutative \textsc{Laurent} polynomial ring~$L_{s}$ in $s$
  indeterminates. Given a group homomorphism
  $p \colon \bZ^{s} \rTo \bZ^{t}$ we let
  $p_{!}(C) = C \tensor_{L_{s}} L_{t}$ denote the resulting induced
  complex over the \textsc{Laurent} polynomial ring~$L_{t}$ in $t$
  indeterminates. We prove that the \textsc{Betti} number jump loci,
  that is, the sets of those homomorphisms $p$ such that
  $b_{k} \big( p_{!} (C) \big) > b_{k} (C)$, have a surprisingly
  simple structure. We allow non-unital commutative rings of
  coefficients, and work with a notion of \textsc{Betti} numbers that
  generalises both the usual one for integral domains, and the
  analogous concept involving \textsc{McCoy} ranks in case of
  unital commutative rings.
\end{abstract}


\maketitle

\section*{Introduction}

In this paper we examine the behaviour of ranks of matrices and
\textsc{Betti} numbers of chain complexes over \textsc{Laurent}
polynomial rings in several indeterminates under a linear change of
variables. Here ``rank'' and ``\textsc{Betti} number'' are taken
relative to a prescribed family of ideals in the ground ring; in
spirit, this is close to (and generalises) considering the
\textsc{McCoy} rank of a matrix over a commutative ring rather than
the usual rank in linear algebra. The resulting purely algebraic
``jump loci'' have a particularly simple structure, which is a
surprising feature of the theory; we start with a digression to
demonstrate that, in general, jump loci can be almost arbitrarily
complicated. We then sketch the motivating results of \textsc{Kohno}
and \textsc{Pajitnov}, before we finally start discussing the specific
set-up under consideration.

\subsection*{General jump loci are complicated}

Jump loci are subsets of a moduli space~$M$ encoding precisely which
of the objects parametrised by~$M$ have a certain desirable property,
and which do not. For example, the space of rank one local systems on
a finite $CW$-complex $X$ may be identified with the algebraic variety
\begin{displaymath}
  M = \hom \big( \pi_1(X), \bC^{\times} \big) = \big( \bC^\times
  \big)^{b_1(X)} \times F
\end{displaymath}
with $F$ a discrete finite abelian group. For $k > 0$ let
$\Sigma^k(X)$ be the subset of~$M$ corresponding to those rank one
systems~$\rho$ satisfying $\dim H^k (X, \rho) \geq 1$; this is an
example of (cohomological) jump loci. No general classification of
such jump loci can be given; in fact, for any algebraic subvariety~$Z$
of the algebraic torus $\big(\bC^\times\big){}^n$ and any $k \geq 1$
there is a finite $CW$ complex~$X$ with $M = \big(\bC^\times\big){}^n$
and $\Sigma^k(X) = Z \cup \{\mathbf{1}\}$, see \textsc{Suciu},
\textsc{Yang} and \textsc{Zhao} \cite[Lemma~10.3]{MR3379489},
\textsc{Simpson}~\cite{MR1492538} and
\textsc{Wang}~\cite[Theorem~1.1]{2013arXiv1304.0239W}.

\textsc{Kohno} and \textsc{Pajitnov} considered, in a similar spirit,
the twisted \textsc{Novikov} homology of a finite $CW$ complex. In
contrast to the case of~$\Sigma^k$ described above, the ensuing
(homological) jump loci have a surprisingly simple structure: they are
finite unions of linear subspaces of~$\bC^n$. This striking result
rests on an analysis of jump loci of a purely algebraic nature, which
we hasten to describe in some detail now.

\subsection*{The Kohno-Pajitnov algebraic jump loci}

Let $R$ be a commutative integral domain with unit. A group
homomorphism $p \colon \bZ^{s} \rTo \bZ^{t}$ determines a homomorphism
$p_{*} \colon R[\bZ^{s}] \rTo R[\bZ^{t}]$ of group rings. Given a
bounded chain complex $C$ of finitely generated free
$R[\bZ^{s}]$-modules we obtain the induced chain complex
$p_{!}(C) = C \tensor_{R[\bZ^{s}]} R[\bZ^{t}]$ of finitely generated
free $R[\bZ^{t}]$-modules. Over a group ring $R[\bZ^{n}]$ we have a
meaningful notion of ``rank'' for matrices and free modules, and can
thus define the \textsc{Betti} numbers $b_{k} (C)$ and
$b_{k} \big(p_{!}(C)\big)$.  \textsc{Kohno} and \textsc{Pajitnov}
proved the following result characterising the jump loci of the
\textsc{Betti} numbers with respect to varying the group
homomorphism~$p$:

\begin{theorem}[{\textsc{Kohno} and \textsc{Pajitnov} \cite[Theorem~7.3]{orig}}]
  \label{thm:KP}
  Let $k \in \bZ$ and $q \geq 0$ be given. There exists a finite
  family of proper direct summands $G_i \subset \hom(\bZ^{s}, \bZ)$
  such that the inequality $b_{k} (C) + q < b_{k} \big(p_{!}(C)\big)$
  holds
  if and only if there is an index~$i$ with $p \in G_{i}^{t}
  \subseteq \hom(\bZ^{s}, \bZ)^{t} = \hom (\bZ^{s}, \bZ^{t})$.
\end{theorem}

In the language used before, the
``moduli space'' in question is $M = \hom (\bZ^{s}, \bZ^{t})$. ---
The first step of the proof is to characterise those~$p$
which satisfy $p_{*}(\Delta) = 0$, for
a fixed non-zero $\Delta \in R[\bZ^{s}]$. Next one establishes a
variant of the theorem for jump loci of the rank of matrices,
characterising the condition $\rank \big( p_{*}(A) \big) < \rank (A)$
for a fixed matrix~$A$ with entries in~$R[\bZ^{s}]$. Finally the
actual theorem can be verified by considering the ranks of the
differentials in the chain complex.

\subsection*{Jump loci from families of ideals}

The \textsc{Kohno}-\textsc{Pajitnov} jump loci described above result
from considering vanishing conditions of the type $p_{*}(\Delta) = 0$.
In the present paper, we take the jump to more general conditions,
showing that a much larger class of jump loci has structure as
described in Theorem~\ref{thm:KP}. To this end, let $\cK$ be a
non-empty set of ideals of~$R$ such that\footnote{Such a set~$\cK$ is
  usually called an order ideal in the partially ordered set of ideals
  of~$R$, but this terminology seems less than ideal in the present
  context.} %
if $I \in \cK$ and $J$ is an ideal contained in~$I$, then $J \in \cK$.
With respect to~$\cK$ and~$\Delta$, we formulate the following
condition on~$p$: {\it The ideal of~$R$ generated by the coefficients
  of the element $p_{*}(\Delta) \in R[\bZ^{t}]$ lies in~$\cK$.} This
should be thought of as saying that $p_{*}(\Delta)$ satisfies a
certain property encoded by~\cK. If \cK{} contains a unique
inclusion-maximal ideal~$M$, the property in question is that of being
zero in~$R/M$.

With respect to the ``property''~\cK{} we establish a notion of
\cK-rank of matrices and \cK-\textsc{Betti} numbers for chain
complexes of finitely generated free modules, and characterise their
jump loci. We allow $R$ to be an arbitrary commutative ring which may
even be non-unital. It is surprising that the aforementioned structure
theory can be established in this generality. The price to pay is that
one cannot expect to have {\it proper\/} direct summands~$G_i$ any
more.

Among the many possible sets~\cK{} two deserve special mention. We
assume a unital commutative ring~$R$ for now. If we take $\cK$ to be
the set of all ideals having non-trivial annihilator then the
$\cK$-rank of a matrix (Definition~\ref{def:X-rank}) is precisely the
\textsc{McCoy}-rank of a matrix as introduced in
\cite[\S2]{MR0006150}. If $R$ is an integral domain and $\cK$ consists
of the zero ideal only, we recover the original result of
\textsc{Kohno} and \textsc{Pajitnov} \cite{orig}.

\subsection*{Notation and conventions}

Throughout the paper, $R$ denotes a fixed commutative ring, possibly
non-unital. The notation $I \ideal R$ is used to indicate that $I$ is
an ideal of~$R$, possibly $\{0\}$ or~$R$ itself. We will concern
ourselves with the group rings $R[\bZ^{s}]$ and $R[\bZ^{t}]$; their
elements are written in the form $\sum_{a \in \bZ^{s}} r_{a} x^{a}$
and $\sum_{b \in \bZ^{t}} \rho_{b} y^{b}$, respectively, with almost
all of the coefficients $r_{a}$ and $\rho_{b}$ being zero.  We let
$\bZ^{s*} = \hom (\bZ^{s}, \bZ)$ stand for the $\bZ$-dual
of~$\bZ^{s}$.  The group $\hom(\bZ^{s}, \bZ^{t})$ is often identified
with $\big( \bZ^{s*}\big){}^{t}$. Any homomorphism
$p \in \hom(\bZ^{s}, \bZ^{t})$ induces an $R$-algebra homomorphism
$p_{*} \colon R[\bZ^{s}] \rTo R[\bZ^{t}]$, which maps
$\Delta = \sum_{a \in \bZ^{s}} r_{a} x^{a} \in R[\bZ^{s}]$ to
\begin{equation}
  \label{eq:p_star_Delta}
  p_{*} (\Delta) = \sum_{b \in \bZ^{t}} \Big( \sum_{a \in p\inv(b)}
  r_{a} \Big) \cdot y^{b} \quad \in R[\bZ^{t}] \ .
\end{equation}
The map $p \in \hom(\bZ^{s}, \bZ^{t})$ also induces a map of sets
of matrices
\begin{displaymath}
  p_{*} \colon \M_{m,n}(R[\bZ^{s}]) \rTo \M_{m,n} (R[\bZ^{t}])
\end{displaymath}
by applying the ring homomorphism~$p_{*}$ to each matrix element.

\section{Properties of ideals and modules}
\label{sec:basic}

Suppose that $\cK$ is a set of ideals of~$R$. This set encodes a
``property'' that elements of the group ring $R[\bZ^{n}]$
may or may not posses. The simplest case is that $\cK$ consists of the
zero ideal only, in which case the property in question is
``being~$0$''.

\begin{definition}
  \label{def:prop_ideals}
  \begin{enumerate}
  \item An ideal $I$ of~$R$ is called a {\it \cK-ideal\/} if
    $I \in \cK$.
  \item A subset $X \subseteq R$ is called a {\it \cK-set\/} if the
    ideal $\langle X \rangle$ generated by~$X$ is a \cK-ideal.
  \end{enumerate}
\end{definition}  

\goodbreak

\begin{definition}
  \label{def:prop_modules}
  Let $X$ be subset of the group ring $R[\bZ^{n}]$.
  \begin{enumerate}
  \item We let $iX$ denote the ideal generated by the set of
    coefficients of the elements of~$X$.
  \item We say that $X$ is a {\it $\cK$-set\/} provided that $iX$ is a
    $\cK$-ideal, \ie, provided that the set of coefficients of
    elements of~$X$ is a $\cK$-set.
  \item An $R$-submodule~$M$ of~$R[\bZ^{n}]$ is called a {\it
      $\cK$-module\/} provided it is a $\cK$-set.
  \end{enumerate}
\end{definition}

In general, we will only be interested in sets~$\cK$ which are closed
under taking smaller ideals:

\begin{definition}
  \label{def:hereditary}
  We call $\cK$ {\it hereditary\/} if $\cK$ is non-empty, and if
  $I \in \cK$ and $J \subseteq I$ together imply $J \in \cK$, for any
  $J \ideal R$.
\end{definition}

As mentioned before, two relevant hereditary sets are
\begin{displaymath}
  \cK_{0} = \big\{ \{0\} \big\} \quad \text{and} \quad
  \cK_{1} = \big\{ I \ideal R \,\big|\, \ann_R(I) \neq
  \{0\} \big\} \ ;
\end{displaymath}
the former corresponds to vanishing conditions, the latter relates to
the \textsc{McCoy}-rank of matrices as explained in
\S\ref{sec:jump-loci-matrices} below. If $V$ is a fixed injective
$R$-module we have the hereditary set
\begin{displaymath}
  \cK_{E,V} = \big\{ I \ideal R \,\big|\, E(I) \text{ embeds into~\(V\)}
  \big\} \ ,
\end{displaymath}
where $E(I)$ denotes an injective hull of the $R$-module~$I$; more
generally, if $\mathcal{V}$ is a family of injective modules, we
obtain a hereditary set
\begin{displaymath}
  \cK_{E,\mathcal{V}} = \big\{ I \ideal R \,\big|\, E(I) \text{ embeds
    into~\(V\) for some \(V \in \mathcal{V}\)} \big\} \ .
\end{displaymath}
Any union of hereditary sets is hereditary, and in fact
$\cK_{E,\mathcal{V}} = \bigcup_{V \in \mathcal{V}} \cK_{E, V}$.  Any
subset $X \subseteq R$ containing~$0$ gives rise to the hereditary set
$\cK_{\subseteq X} = \big\{ I \ideal R \,|\, I \subseteq X \big\}$,
and, if $X$ contains at least one element other than~$0$, the
hereditary set
$\cK_{\subsetneqq X} = \big\{ I \ideal R \,|\, I \subsetneqq X
\big\}$.
In particular, we can consider the hereditary set
$\cK_{\subsetneq J(R)}$ in case the \textsc{Jacobson} radical~$J(R)$
of~$R$ is non-trivial. For a unital ring~$R$, this is the hereditary
set of all superfluous ideals different from~$J(R)$, where $I$ is {\it
  superfluous\/} if $I + J = R$ implies $J = R$, for all $J \ideal R$.

Before we construct more examples of hereditary sets, we prove a basic
result relating the property of ``being a \cK-set'' for a set
$X \subseteq R[\bZ^{s}]$ and its image~$p_{*}(X)$, for
$p \in \hom(\bZ^{s}, \bZ^{t})$. Those~$p$ where
$p_{*}(X)$ is a \cK-set form the ``jump loci'' from the title of the
paper.

\begin{lemma}
  \label{lem:X_through_hom}
  Let \cK{} be a hereditary set of ideals of~$R$, and let
  $X \subseteq R[\bZ^{s}]$ be a subset. For any
  $p \in \hom(\bZ^{s}, \bZ^{t})$, if $X$ is a $\cK$-set then so
  is~$p_{*}(X)$.
\end{lemma}

\begin{proof}
  Let $I$ be the ideal generated by the coefficients of elements
  of~$X$; we have $I \in \cK$ by hypothesis on~$X$. Let $J$ be the
  ideal generated by the coefficients of elements of~$p_{*}(X)$. The
  elements of $p_{*}(X)$ are of the form $p_{*}(\Delta)$, for
  $\Delta \in M$, and by formula~(\ref{eq:p_star_Delta}) the
  coefficients of~$p_{*}(\Delta)$ are sums of coefficients
  of~$\Delta$. But the latter coefficients are elements of~$I$, hence
  so are the former. That is, the generators of~$J$ are elements
  of~$I$ whence $J \subseteq I$. As $\cK$ is hereditary we conclude
  $J \in \cK$ so that $p_{*}(X)$ is a $\cK$-set as claimed.
\end{proof}

\goodbreak

As our set~$\cK$ always contains the zero ideal, we also have the
following:

\begin{lemma}
  \label{lem:augmentation}
  If $X$ is a subset of the augmentation ideal $I_{s} = \ker 0_*$
  of~$\bZ^{s}$, then $0_{*}(X)$ is a \cK-set. \qed
\end{lemma}

\begin{remark}
  If the hereditary set~\cK{} contains a unique maximal element~$J$,
  then the conditions ``being a \cK-module'' can be transformed into
  the annihilation condition ``being trivial'' by replacing the ground
  ring~$R$ with~$R/J$.
\end{remark}

\subsection*{Hereditary sets and filters of ideals}

A {\it filter\/} of ideals is a non-empty set~$\cF$ of ideals of~$R$
such that $J \in \cF$ and $I \supseteq J$ together imply $I \in \cF$.
There is an intimate connection between hereditary sets of ideals and
filters of ideals:

\begin{proposition}
  \label{prop:filters_and_ideals}
  \begin{enumerate}
  \item Every filter~$\cF$ of ideals determines a hereditary set of
    ideals given by
    $\cF' = \big\{ I \ideal R \,|\, \ann_R(I) \in \cF \big\}$, and the
    assignment $\cF \mapsto \cF'$ is inclusion-reversing.
  \item Every hereditary set~$\cK$ of ideals determines a filter of
    ideals given by
    $\cK' = \big\{ J \ideal R \,|\, \ann_R(J) \in \cK \big\}$, and the
    assignment $\cK \mapsto \cK'$ is inclusion-reversing. \qed
  \end{enumerate}
\end{proposition}

Hereditary sets of ideals can be constructed with the aid of
Proposition~\ref{prop:filters_and_ideals}, for example from the filter
$ \cF_{\mathrm{ess}} = \big\{ J \ideal R \,\big|\, J \text{ is an
  essential ideal} \big\}$,
where $J$ is {\it essential\/} if $J \cap I = \{0\}$ implies
$I = \{0\}$, for all $I \ideal R$. Using
Proposition~\ref{prop:filters_and_ideals} twice, any hereditary
set~$\cK$ gives rise to another hereditary set $\cK'' \supseteq \cK$;
specifically, $\cK_{0}' \subseteq \cF_{\mathrm{ess}}$ and hence
$\cK_{0}'' \supseteq \cF_{\mathrm{ess}}'$.

Any subset $X$ of~$R$ defines a filter
$\cF_{\supseteq X} = \big\{ J \ideal R \,|\, J \supseteq X \big\}$,
and, if $X \neq R$, also the filter
$\cF_{\supsetneqq X} = \big\{ J \ideal R \,|\, J \supsetneqq X
\big\}$.
Finally, we observe that any intersection of filters of ideals is
again a filter.

\section{Partition subgroups of~$\hom(\bZ^{s}, \bZ^{t})$}

\begin{definition}
  Suppose $\pi = (\pi_{1},\, \pi_{2},\, \cdots,\, \pi_{k})$ is a
  partition of a subset of~$\bZ^{s}$, so that
  $\pi = \coprod_{j=1}^{k} \pi_{j} \subseteq \bZ^{s}$. (We allow the
  case $\pi = \emptyset$ and $k=0$ here.) To~$\pi$ we associate the
  abelian group
  \begin{displaymath}
    H(\pi) = \{ p \in \bZ^{s*} \,|\, \forall j=1,\, 2,\, \cdots,\, k :
    \ p|_{\pi_{j}} \text{ is constant}\}
  \end{displaymath}
  called the {\it partition subgroup associated to~$\pi$}. For a
  set~$P$ of partitions of (possibly distinct) subsets of~$\bZ^{s}$,
  we define
  \begin{displaymath}
    H(P) = \bigcap_{\pi \in P} H(\pi) \ ,
  \end{displaymath}
  and call $H(P)$ the {\it partition subgroup associated to~$P$}.
\end{definition}

\begin{lemma}
  \label{lem:summand}
  Suppose $P$ is a set of partitions of subsets of~$\bZ^{s}$.
  \begin{enumerate}
  \item If at least one part~$\pi_j$ of at least one partition
    $\pi \in P$ has at least two elements, then $H(P)$ is a proper
    subgroup of~$\bZ^{s*}$.
  \item The subgroup $H(P)$ is contained in~$\bZ^{s*}$ as a direct
    summand.
  \end{enumerate}
\end{lemma}

\begin{proof}
  Part~(1) is trivial. Let us prove~(2). In view of the canonical
  short exact sequence
  \begin{displaymath}
    0 \rTo H(P) \rTo^{\subseteq} \bZ^{s*} \rTo \bZ^{s*}/H(P) \rTo 0
  \end{displaymath}
  it is enough to show that $\bZ^{s*}/H(P)$ is torsion-free. For
  then $\bZ^{s*}/H(P)$ is free abelian, whence the short
  exact sequence splits.

  So let $q \in \bZ^{s*}$ be such that $[q] \in \bZ^{s*}/H(P)$ is
  torsion. Then there exists a natural number $n \geq 1$ with
  \begin{displaymath}
    [nq] = n \cdot [q] = 0 \in \bZ^{s*}/H(P) \ ,
  \end{displaymath}
  that is, $nq \in H(P)$. By definition of~$H(P)$ this means that the
  homomorphism $nq$ is constant on each part~$\pi_{j}$ of each
  partition $\pi \in P$, which implies that $q$ has the same
  property. Consequently, $q \in H(P)$ and thus $[q] = 0$.
\end{proof}

Let
$\Delta = \big\{ \Delta^{(1)},\, \Delta^{(2)},\, \cdots,\,
\Delta^{(d)} \big\}$,
for $d \geq 0$, be a finite subset of~$R[\bZ^{s}]$. To fix notation,
we write
\begin{displaymath}
  \Delta^{(i)} = \sum_{a \in \bZ^{s}} r^{(i)}_{a} x^{a} \qquad
  \text{(for \(1 \leq i \leq d\)).}
\end{displaymath}
Let $\pi = (\pi_{1},\, \pi_{2},\, \cdots,\, \pi_{k})$ be a partition
of $\supp(\Delta) = \bigcup_{i=1}^{d} \supp(\Delta^{(i)})$. We define
ring elements
\begin{equation}
\label{eq:generators}
\Delta^{(i)}_{j} = \sum_{a \in \pi_{j}} r^{(i)}_{a} \ \in R 
  \qquad \text{(for \(1 \leq j \leq
  k\) and \(1 \leq i \leq d\)),}
\end{equation}
and denote the ideal generated by these elements by
\begin{displaymath}
\Delta_{\pi} = \big\langle \Delta^{(i)}_{j} \,\big|\, 1 \leq j \leq
k,\ 1 \leq i \leq d \big\rangle \quad \ideal R \ .
\end{displaymath}

\begin{definition}
  Let $\cK$ be a hereditary set of ideals of~$R$,
  cf.~Definition~\ref{def:hereditary}.  A {\it $\cK$-partition\/} is a
  partition $\pi$ of $\supp(\Delta)$ such that $\Delta_{\pi}$ is a
  $\cK$-ideal.
\end{definition}

\begin{lemma}
  \label{lem:KP_proper}
  Let $\cK$ be a hereditary set of ideals of~$R$. Suppose that
  $i\Delta \notin \cK$, that is, $i\Delta$ is not a \cK-ideal. Suppose
  $P$ is a set of partitions of subsets of~$\bZ^{s}$ which contains
  at least one $\cK$-partition~$\pi$ of~$\supp(\Delta)$. Then $H(P)$
  is a proper subgroup of~$\bZ^{s*}$.
\end{lemma}

\begin{proof}
  Since $i\Delta \notin \cK$ we have $i\Delta \neq \{0\}$, so the set
  $\Delta$ contains a non-zero element whence
  $\supp(\Delta) \neq \emptyset$. Let
  $\tau = (\tau_{1},\, \tau_{2},\, \cdots,\, \tau_{k}) \in P$ be any
  partition of~$\supp(\Delta)$. If all parts~$\tau_{j}$ of~$\tau$ are
  singletons then the generators of~$\Delta_{\tau}$, as specified
  in~\eqref{eq:generators}, are precisely the generators of~$i\Delta$
  so that $\Delta_{\tau} = i\Delta \notin \cK$. That is, such a~$\tau$
  is not a \cK-partition. But the stipulated partition~$\pi$ {\it
    is\/} a \cK-partition; it follows that at least one part of~$\pi$
  must have at least two elements. Now Lemma~\ref{lem:summand}
  applies, assuring us that $H(P)$ is a proper direct summand
  of~$\bZ^{s*}$.
\end{proof}

\section{Jump loci for modules}

In this section, $\cK$ denotes a fixed hereditary set of ideals of~$R$
in the sense of Definition~\ref{def:hereditary}.

\begin{proposition}
  \label{prop:jump_module}
  Let $M$ be a non-trivial,
  finitely generated $R$-submodule of the group ring
  $R[\bZ^{s}]$. There exist partition subgroups
  $G_{1},\, G_{2},\, \cdots,\, G_{\ell} \subseteq \bZ^{s*}$, where
  $ \ell \geq 0$, such that for every $p \in \hom(\bZ^{s}, \bZ^{t})$,
  \begin{displaymath}
    p_{*}(M) \text{ is a \(\cK\)-module} %
    \qquad \Longleftrightarrow \qquad %
    p \in \bigcup_{j=1}^{\ell} G_{j}^{t} \ .
  \end{displaymath}
  Here
  $G_{j}^{t} = \bigoplus_{1}^{t} G_{j} \subseteq \big(
  \bZ^{s*}\big){}^{t} = \hom(\bZ^{s}, \bZ^{t})$.
  --- More precisely, writing
  \[\Delta = \big\{ \Delta^{(1)},\, \Delta^{(2)},\, \cdots,\,
  \Delta^{(d)} \big\} \subseteq M\]
  for a finite generating set of the $R$-module~$M$, the number~$\ell$
  is the number of $\cK$-partitions of $\supp(\Delta)$, and the
  groups~$G_j$ are the partition subgroups $H(\pi)$ associated to
  these partitions.

  If $M$ is not a $\cK$-module then the groups $G_j$ are proper
  subgroups of~$\bZ^{s*}$ so that
  $\bigcup_{j=1}^{\ell} G_{j}^{t} \neq \hom(\bZ^{s}, \bZ^{t})$. We
  have $\ell > 0$ if and only if $0_{*}(M)$ is a \cK-module, which
  certainly is the case if $M$ is contained in the augmentation ideal
  of~$R[\bZ^{s}]$.
\end{proposition}

For ease of reading we delegate the main step of the proof to the
following Lemma, which is based on the argument from \cite[Theorem~7.3]{orig}. Write $p = (p_{1},\, p_{2},\, \cdots,\, p_{t})$, with
each component $p_{i} \colon \bZ^{s} \rTo \bZ$ being an element
of~$\bZ^{s*}$.

\begin{lemma}
  \label{lem:main}
  Let $M$ and~$\Delta$ be as in Proposition~\ref{prop:jump_module}.
  The $R$-submodule $p_{*}(M)$ of~$R[\bZ^{t}]$ is a $\cK$-module if
  and only if there exists a $\cK$-partition~$\pi$ of $\supp(\Delta)$
  such that $p$~is constant on each part~$\pi_{j}$ of~$\pi$, that is,
  such that $p \in H(\pi)^{t}$.
\end{lemma}

\begin{proof}
  Let us prove the ``if'' implication first. Suppose that there is a
  $\cK$\nbd-partition~$\pi$ of~$\supp(\Delta)$ with $p \in H(\pi)^t$;
  this last condition is equivalent to $p$ being constant on each
  part~$\pi_j$ of~$\pi$. As $p_{*}(M)$ is generated (as an $R$-module)
  by the set $p_{*}(\Delta)$, the ideal $ip_{*}(M)$ of~$R$ equals
  $ip_{*}(\Delta)$.  We write
  \begin{displaymath}
    p_{*} (\Delta^{(i)}) = \sum_{b \in \bZ^{t}} \rho^{(i)}_b y^b = \sum_{b \in
      \bZ^{t}} \Big( \sum_{a \in p\inv(b)} r^{(i)}_{a} \Big) \cdot y^{b}
  \end{displaymath}
  where $\rho^{(i)}_b = \sum_{a \in p\inv(b)} r^{(i)}_{a}$. As $\pi$
  is a partition of~$\supp(\Delta)$, and as $p$ is constant on each
  part of~$\pi$, we have the equality
  \begin{equation}
    \label{eq:2}
    \rho^{(i)}_{b} = \sum_{j} \sum_{\substack{a \in \pi_{j} \\ {a \in p\inv(b)}}}
    r^{(i)}_{a} = \sum_{\substack{j \\ p|_{\pi_{j}}\equiv b}} \sum_{a
      \in \pi_j} r^{(i)}_{a} = \sum_{\substack{j \\ p|_{\pi_{j}}\equiv
        b}} \Delta^{(i)}_{j} \ .
  \end{equation}
  As $\pi$ is a $\cK$-partition, the ideal
  $\Delta_{\pi} = \big\langle \Delta^{(i)}_{j} \,|\, 1 \leq j \leq k,\
  1 \leq i \leq d \big\rangle$
  is an element of~$\cK$. It contains all the~$\Delta^{(i)}_j$ and
  hence all the~$\rho^{(i)}_b$, by~\eqref{eq:2}. As $\cK$~is
  hereditary the ideal generated by the~$\rho^{(i)}_b$ thus also lies
  in~$\cK$. But this ideal is precisely $ip_{*}(M)$, as observed
  above, so $p_{*}(M)$ is a $\cK$-module as desired.

  \medbreak

  To show the reverse implication suppose that $p_{*} (M)$ is a
  $\cK$-module. Written more explicitly (using the notation from the
  previous paragraph) this means that the ideal
  $ip_{*}(M) = ip_{*}(\Delta)$ generated by the elements
  $\rho^{(i)}_b = \sum_{a \in p\inv(b)} r_{a}$, for $b \in \bZ^{t}$
  and $1 \leq i \leq d$, lies in~$\cK$. The requisite partition~$\pi$
  of~$\supp(\Delta)$ is defined by declaring those intersections
  $\supp(\Delta) \cap p\inv(b)$ which are non-empty to be the parts
  of~$\pi$, where $b$ varies over all of~$\bZ^{t}$. By construction,
  each component~$p_{i}$ of~$p$ is constant on each part of~$\pi$; on
  the part corresponding to
  $b = (b_{1},\, b_{2},\, \cdots,\, b_{t}) \in \bZ^{t}$, the
  component~$p_{i}$ takes the constant value~$b_{i}$. The
  corresponding elements~$\Delta^{(i)}_j$ are exactly the elements of
  the form~$\rho^{(i)}_b$, so
  $\Delta_{\pi} = ip_{*} (\Delta) = ip_{*}(M)$ is a $\cK$-ideal. We
  have thus shown that $\pi$ is a $\cK$-partition and
  $p \in H(\pi)^{t}$, as required.
\end{proof}

\begin{proof}[Proof of Proposition~\ref{prop:jump_module}]
  We re-state the conclusion of Lemma~\ref{lem:main}:
  \begin{displaymath}
    p_{*}(M) \text{ is a \(\cK\)-module} \qquad \Longleftrightarrow \qquad p
    \in \bigcup_{\pi} H(\pi)^{t} \ , 
  \end{displaymath}
  the union extending over the finite set of all
  $\cK$-partitions~$\pi$ of $\supp(\Delta)$. Up to renaming the groups
  occurring on the right-hand side, this is the condition stated in the
  Proposition. We have verified that each $H(\pi)$ is a direct summand
  of~$\bZ^{s*}$ in Lemma~\ref{lem:summand} above. 

  \medbreak

  In case $M$ is not a \cK-module we know from
  Lemma~\ref{lem:KP_proper} that the partition subgroups $H(\pi)$ are
  proper subgroups of~$\bZ^{s*}$. We have $\ell > 0$ if and only if
  $0 \in \bigcup_{j=1}^{\ell} G_{j}^{t}$ if and only if $0_{*}(M)$ is
  a \cK-module. If $M$ is contained in the augmentation ideal $I_{s} =
  \ker (0_*)$ of~$R[\bZ^{s}]$ then 
  $0_{*}(M) = \{0\}$ is a \cK-module so that the union
  $ \bigcup_{j=1}^{\ell} G_{j}^{t}$ must contain~$0$; this forces
  $\ell > 0$. This finishes the proof of
  Proposition~\ref{prop:jump_module}.
\end{proof}

\begin{corollary}
  \label{cor:finite_modules}
  Let $M_1,\, M_2,\, \cdots,\, M_k$ be finitely generated
  $R$-submodules of~$R[\bZ^{s}]$. There are partition subgroups
  $G_{1},\, G_{2},\, \cdots,\, G_{\ell} \subseteq \bZ^{s*}$, where
  $\ell \geq 0$, such that for all $p \in \hom(\bZ^{s}, \bZ^{t})$,
  \begin{displaymath}
    \forall i=1,\, 2,\,\cdots,\, k : \ p_{*}(M_i) \text{ %
      is a \(\cK\)-module } \qquad \Longleftrightarrow \qquad p \in
    \bigcup_{j=1}^{\ell} G_{j}^{t} \ .
  \end{displaymath}
  If at least one of the modules~$M_{j}$ is not a \cK-module the
  groups $G_j$ are proper subgroups of~$\bZ^{s*}$ so that
  $\bigcup_{j=1}^{\ell} G_{j}^{t} \neq \hom(\bZ^{s},
  \bZ^{t})$. Moreover, $\ell > 0$ if and only if all the modules
  $0_*(M_j)$, for $1 \leq j \leq k$, are \cK-modules, which is
  certainly the case if all of the modules~$M_{j}$ are contained in
  the augmentation ideal~$I_{s}$ of~$R[\bZ^{s}]$.
\end{corollary}

\begin{proof}
  Let $\Delta_{(i)}$ be a finite generating set for the
  $R$-module~$M_i$. By Proposition~\ref{prop:jump_module}, we know
  that the statement
  \begin{displaymath}
    \forall i=1,\, 2,\,\cdots,\, k : \ p_{*}(M_i) \text{ %
      is a \(\cK\)-module}
  \end{displaymath}
  holds if and only if there exist $\cK$-partitions $\pi_{(i)}$ of
  $\supp(\Delta_{(i)})$, for $1 \leq i \leq k$, such that all
  components of~$p$ are constant on all parts of $\pi_{(i)}$. The
  latter condition is equivalent to saying
  $p \in \bigcap_{i=1}^{k} H(\pi_{(i)})^{t}$. So the requisite finite
  family of subgroups of~$\bZ^{s*}$ is given by the family of
  intersections
  $\bigcap_{i=1}^{k} H(\pi_{(i)}) = H \big( \{\pi_{(i)} \,|\, 1 \leq i
  \leq k\} \big)$,
  with the $\pi_{(i)}$ ranging independently over all $\cK$-partitions
  of $\supp(\Delta_{(i)})$. --- The additional properties follow as in
  Proposition~\ref{prop:jump_module}.
\end{proof}

\section{Jump loci for the rank of matrices}
\label{sec:jump-loci-matrices}

Let $A$ be a matrix with entries in~$R[\bZ^{s}]$. We apply the results
of the previous section to analyse the dependence of the rank of the
matrix $p_{*}(A)$ from the homomorphism
$p \in \hom(\bZ^{s}, \bZ^{t})$. First, we need to clarify what we mean
by ``rank''.

\subsection*{The $\cK$-rank of a matrix}
\label{sec:K-rank}

We denote by~$\cK$ a hereditary set of ideals of~$R$ in the sense of
Definition~\ref{def:hereditary}, with $R$ an arbitrary commutative
ring. Given any matrix~$A$ we write $|A|$ for the set of its
entries. Now let $A$ be specifically an $m \times n$-matrix with
entries in~$R[\bZ^{s}]$, and let $z$ be a $k$-minor of~$A$, that is,
the determinant of a square sub-matrix of~$A$ of size~$k$. (We remark
here that determinants are defined in the usual fashion, {\it via\/} a
sum indexed by the symmetric group or, equivalently, using
\textsc{Laplace} expansion and induction on~$k$, and that determinants
have all the usual properties. See the discussion in \S2
of~\cite{MR1563966}.) Let $z'$ be the minor of~$p_{*}(A)$
corresponding to~$z$, for some fixed $p \in \hom(\bZ^{s},
\bZ^{t})$.
Then we have $p_{*}(z) = z'$, as
$p_{*} \colon R[\bZ^{s}] \rTo R[\bZ^{t}]$ is a ring homomorphism. We
also have $p_{*} \big( |A| \big) = \big|p_{*} (A)\big|$.

\begin{definition}
  \label{def:X-rank}
  We say that {\it $A$ has $\cK$-rank~$0$\/} and write
  $\rank_{\cK}(A) = 0$ if $|A|$ is a $\cK$-set (that is, if the ideal
  generated by the coefficients of the entries of~$A$ is a
  $\cK$-ideal). Otherwise, the {\it \cK-rank of~$A$} is the maximal
  integer $k = \rank_{\cK}(A) > 0$ such that the set of $k$-minors
  of~$A$ is {\it not\/} a $\cK$-set (that is, the ideal generated by
  the coefficients of the $k$-minors of~$A$ is {\it not\/} a
  $\cK$-ideal).
\end{definition}

The following Lemma sheds some light on this definition.

\begin{lemma}
  \label{lem:meaningful}
  If the set of $k$-minors of~$A$ is a $\cK$-set, then so is
  the set of $(k+1)$-minors.
\end{lemma}

\begin{proof}
  By the familiar expansion formula of determinants, each minor of
  size~$k+1$ is a linear combination of minors of size~$k$, and the
  claim follows as \cK{} is a hereditary property.
\end{proof}

If $R$ is a (commutative) {\it unital\/} ring, the $\cK_{1}$-rank
coincides with the rank of matrices over~$R[\bZ^{s}]$ as considered by
\textsc{McCoy} \cite[\S2]{MR0006150}; we will show this in
Proposition~\ref{prop:McCoy=K_1} below. If $R$ is a field, the
$\cK_{0}$-rank coincides with the usual rank from linear algebra.

\begin{proposition}
  \label{prop:rank_drops}
  For each $p \in \hom(\bZ^{s}, \bZ^{t})$ we have the inequality
  \begin{displaymath}
    \rank_{\cK} \big(p_{*}(A) \big) \leq \rank_{\cK}(A) \ .
  \end{displaymath}
\end{proposition}

\begin{proof}
  Let $k = 1 + \rank_{\cK} (A)$. The set of $k \times k$-minors of~$A$
  is a $\cK$-set, by definition of rank, hence so is its image
  under~$p_{*}$ by Lemma~\ref{lem:X_through_hom}. But this image is
  precisely the set of $k \times k$-minors of~$p_{*}(A)$, whence
  $p_{*}(A)$ must have $\cK$-rank strictly less than~$k$.
\end{proof}

\subsection*{The McCoy-rank of a matrix}

Let $S$ denote a commutative {\it unital\/} ring, and let~$A$ be an
$m \times n$-matrix with entries in~$S$.

\begin{definition}
  We say that the matrix $A$~has \textit{\textsc{McCoy}-rank~$0$},
  written $\mcrank(A)=0$, if there exists a non-zero element of~$S$
  annihilating every entry of~$A$. Otherwise, the
  \textit{\textsc{McCoy}-rank of~$A$\/} is the maximal integer
  $k = \mcrank(A)$ such that the set of $k$-minors of~$A$ is {\it
    not\/} annihilated by a non-zero element of~$S$.
\end{definition}

In complete analogy to Lemma~\ref{lem:meaningful} one can show that
{\it if the set of $k$-minors of~$A$ is annihilated by a non-zero
  element of~$S$, then so is the set of $(k+1)$-minors}. The relevance
of the \textsc{McCoy}-rank is revealed by the following result:

\begin{theorem}[{\textsc{McCoy} \cite[Theorem~1]{MR0006150}}]
  The homogeneous system of $m$ linear equations in $n$ variables
  represented by the matrix~$A$ has a non-trivial solution over~$S$ if
  and only if $\mcrank(A) < n$.\qed
\end{theorem}

We are of course mainly interested in the special case
$S = R[\bZ^{s}]$ with $R$ a commutative unital ring, in which case we
also have the notion of $\cK_{1}$-rank at our disposal. Note that we
have the inequality
\begin{displaymath}
  \rank_{\cK_{1}} (A) \geq \mcrank(A) \ ;
\end{displaymath}
indeed, if the set of coefficients of the $k$-minors of~$A$ is
annihilated by a non-zero element~$y$ of~$R$ so that
$\rank_{\cK_{1}}(A) < k$, then $y$ also annihilates the $k$-minors
themselves so that $\mcrank(A) < k$ as well.

Conversely, assume that the (finite) set $X \subseteq R[\bZ^{s}]$ of
$k$-minors of~$A$ is annihilated by a non-zero element
$x \in R[\bZ^{s}] \iso R[X_{1}^{\pm 1},\, X_{2}^{\pm 1},\, \cdots,\,
X_{s}^{\pm 1}]$.
For every $\ell \in \bZ$ the element
$u = X_1^{\ell} X_2^{\ell} \cdots X_{s}^{\ell}$ is a unit; we choose
$\ell \gg 0$ such that $X' = uX$ and $x' = ux$ lie in the monoid ring
\begin{displaymath}
P = R[\bN^{s}] \iso R[X_{1},\, X_{2},\, \cdots,\, X_{s}] \ .
\end{displaymath}
The ideal $\langle X' \rangle \ideal P$ generated by~$X'$ is
annihilated by $x' \neq 0$. By a result of \textsc{McCoy}
\cite[Theorem]{MR0082486} we can thus find a non-zero element
$y \in R$ annihilating $X' = uX$. But then $y$ also annihilates
$u^{-1} X' = X$, and as $y$ is an element of~$R$ this implies that $y$
annihilates the coefficients of the elements of~$X$ individually. This
yields the inequality $\rank_{\cK_{1}} (A) \leq \mcrank(A)$. --- We
have shown:

\begin{proposition}
  \label{prop:McCoy=K_1}
  If $R$ is a commutative unital ring, the two quantities
  $\rank_{\cK_{1}}(A)$ and $\mcrank(A)$ coincide.\qed
\end{proposition}

\subsection*{Jump loci}

As before, we denote by~$\cK$ a hereditary set of ideals of~$R$ in the
sense of Definition~\ref{def:hereditary}. Let $q \geq 0$, and let
$A_{1},\, A_{2},\, \cdots,\, A_{k}$ be matrices (of various sizes)
over~$R[\bZ^{s}]$, with \cK-ranks $r_{i} = \rank_{\cK}(A_{i})$. We want
to characterise those $p \in \hom(\bZ^{s}, \bZ^{t})$ such that
\begin{equation}
  \label{eq:1}
  \forall i = 1,\, 2,\, \cdots,\, k : \ \rank_{\cK}\big( p_{*}(A_{i})
  \big) < \rank_{\cK}(A_{i}) - q \ .
\end{equation}
Of course such $p$ cannot exist if $r_i \leq q$ for some~$i$.
Otherwise, we see from Lemma~\ref{lem:meaningful} that~\eqref{eq:1} is
true if and only if for all indices~$i$, the image of the set of
$(r_{i}-q)$-minors of~$A_{i}$ under~$p_{*}$ is a $\cK$-set. Writing
$M_i$ for the $R$-submodule of~$R[\bZ^{s}]$ generated by the
$(r_{i}-q)$-minors of~$A_i$, this is equivalent to saying that the
modules $p_{*}(M_{i})$ are $\cK$-modules. Now
Corollary~\ref{cor:finite_modules} applies. Note that, by definition
of rank, the modules~$M_i$ are {\it not\/} $\cK$-modules. We have
shown:

\begin{theorem}
  \label{thm:rank_matrix}
  Let $A_{1},\, A_{2},\, \cdots,\, A_{k}$ be matrices (of various
  sizes) over the ring~$R[\bZ^{s}]$, and let $q \geq 0$. There are a number
  $\ell \geq 0$ and direct summands
  $G_{1},\, G_{2},\, \cdots,\, G_{\ell}$ of~$\bZ^{s*}$ such that for
  $p \in H$,
  \begin{multline*}
    \forall i=1,\, 2,\, \cdots,\, k : \ \rank_{\cK}\big( p_{*}(A_{i})
    \big) 
    < \rank_{\cK}(A_{i}) - q \\
    \Longleftrightarrow \qquad p \in
    \bigcup_{j=1}^{\ell} G_{j}^{t} \ . \qquad \quad
  \end{multline*}
  If there is an index~$i$ such that $r_{i} \leq q$ then $\ell =
  0$.
  Otherwise, the groups $G_j$ are proper subgroups of~$\bZ^{s*}$ so
  that $\bigcup_{j=1}^{\ell} G_{j}^{t} \neq \hom(\bZ^{s}, \bZ^{t})$,
  and moreover $\ell > 0$ if the set of $(r_{i}-q)$-minors of~$A_{i}$
  is contained in the augmentation ideal~$I_{s}$ of~$R[\bZ^{s}]$ for
  all~$i$. \qed
\end{theorem}

\section{\cK-Betti numbers and their jump loci}
\label{sec:ck-betti-numbers}

Let $C$ be a chain complex (possibly unbounded) consisting of finitely
generated free based $R[\bZ^{s}]$-modules; more precisely, we suppose
that $C_k = (R[\bZ^{s}])^{r_k}$ for certain integers $r_k \geq 0$. We
call $r_k$ the rank of~$C_k$; if $R$ is unital, the (commutative) ring
$R[\bZ^{s}]$ has IBN and thus the isomorphism type of~$C_k$
determines~$r_k$ uniquely.

Our differentials lower the degree, $d_{k} \colon C_{k} \rTo C_{k-1}$,
and we assume that each map $d_k$ is given by multiplication by a
matrix~$D_k$ with entries in~$R[\bZ^{s}]$. (For non-unital~$R$ this
potentially restricts the set of allowed differentials.) We define the
\cK-rank of the homomorphism $d_k \colon C_{k} \rTo C_{k-1}$, denoted
$\rank_{\cK} (d_{k})$, to be the \cK-rank $\rank_{\cK}(D_{k})$ of the
matrix~$D_k$, as defined previously.

\begin{definition}
  The $k$th \cK-\textsc{Betti} number $b_{k}^{\cK} = b_{k}^{\cK}(C)$
  of~$C$ is
  \begin{displaymath}
    b_{k}^{\cK} (C) = r_k - \rank_{\cK} (d_{k}) - \rank_{\cK} (d_{k+1}) \
    .
  \end{displaymath}
\end{definition}

Given a homomorphism $p \in \hom(\bZ^{s}, \bZ^{t})$, we define a new
chain complex $p_{!}(C)$ by setting $p_{!}(C)_k = (R[\bZ^{t}])^{r_k}$,
equipped with differentials denoted $p_{!}(d_k)$ given by the matrices
$p_{*}(D_k)$. In case the multiplication map
$R[\bZ^{s}] \tensor_{R[\bZ^{s}]} R[\bZ^{t}] \rTo R[\bZ^{t}]$,
$x \tensor y \mapsto p_{*}(x) \cdot y$ is an isomorphism (this
happens, for example, if $R$ is unital), we have
$p_{!}(C) = C \tensor_{R[\bZ^{s}]} R[\bZ^{t}]$.  --- In general we
have $b_{k}^{\cK} \big( p_{!}(C) \big) \geq b_{k}^{\cK} (C)$ by
Proposition~\ref{prop:rank_drops}, with strict inequality if and only
if at least one of the strict inequalities
$\rank_{\cK} \big(p_{!}(d_{k}) \big) < \rank_{\cK} (d_{k})$ and
$\rank_{\cK} \big(p_{!}(d_{k+1}) \big) < \rank_{\cK} (d_{k+1})$ is
satisfied.

More generally, given $q \geq 0$ we want to characterise those
$p \in \hom(\bZ^{s}, \bZ^{t})$ such that
\begin{equation}
\label{eq:inc_Betti}
  b_{k}^{\cK} \big( p_{!}(C) \big) > b_{k}^{\cK} (C) + q \ .
\end{equation}
This happens if and only if applying~$p_{!}$ lowers the \cK-rank
of~$d_{k}$ by at least $q+1-j$, and lowers the \cK-rank of~$d_{k+1}$ by
at least $j$, for some $j$ in the range $0 \leq j \leq q+1$. Stated
more formally:

\begin{lemma}
  The inequality~\eqref{eq:inc_Betti} holds if and only if there
  exists a number~$j$ with $0 \leq j \leq q+1$ such that
  \begin{subequations}
    \begin{gather}
      \rank_{\cK} \big( p_{!}(d_{k}) \big) \leq \rank_{\cK} (d_{k}) -
      (q+1-j) \label{eq:a} \\
      \noalign{\noindent and} \rank_{\cK} \big( p_{!}(d_{k+1}) \big)
      \leq \rank_{\cK} (d_{k+1}) -j \ . \label{eq:b}
    \end{gather}
  \end{subequations}
\end{lemma}

\begin{proof}
  If such $j$ exists, then we have indeed
  \begin{align*}
    b_{k}^{\cK} \big( p_{!}(C) \big) &= r_{k} - \rank_{\cK} \big(
                                      p_{!}(d_{k}) \big) - \rank_{\cK}
                                      \big( p_{!}(d_{k+1}) \big)\\ 
    & \geq r_{k} - \big( \rank_{\cK} (d_{k}) - (q+1-j) \big) - \big(
      \rank_{\cK} (d_{k+1}) -j \big) \\  
    & = r_{k} - \rank_{\cK} (d_{k}) - \rank_{\cK} (d_{k+1}) + q+1 \\
    & > b_{k}^{\cK}(C) + q \ .
  \end{align*}
  For the converse, suppose that for each~$j$ at least one of the
  inequalities \eqref{eq:a} and~\eqref{eq:b} is violated. Inequality
  \eqref{eq:b} holds for $j=0$, by Proposition~\ref{prop:rank_drops}; let
  $m \leq q+1$ be maximal such that \eqref{eq:b} is true for
  $0 \leq j \leq m$. We must have $m \leq q$; otherwise \eqref{eq:a}
  must be violated for $j=m=q+1$, resulting in the inequality
  $\rank_{\cK} \big( p_{!}(d_{k}) \big) > \rank_{\cK} (d_{k})$ which is
  known to be nonsense, by Proposition~\ref{prop:rank_drops} again.

  We have established that \eqref{eq:b} holds for $j=m \leq q$ but is
  violated for $j=m+1$, that is, we know
  \begin{gather*}
    \rank_{\cK} \big( p_{!}(d_{k+1}) \big) \leq \rank_{\cK} (d_{k+1}) -m
    \\ % 
    \noalign{\noindent and} %
    \rank_{\cK} \big( p_{!}(d_{k+1}) \big) >  \rank_{\cK} (d_{k+1}) -
    (m+1)
  \end{gather*}
  which yields the equality
  \begin{displaymath}
    \rank_{\cK} \big( p_{!}(d_{k+1}) \big) = \rank_{\cK} (d_{k+1}) -m \
    .
  \end{displaymath}
  As \eqref{eq:b} holds for $j=m$ we know that \eqref{eq:a} must be
  false for $j=m$. This, together with the previous equality, provides
  the estimate
  \begin{align*}
    b_{k}^{\cK} \big( p_{!}(C) \big) &= r_{k} - \rank_{\cK} \big(
                                      p_{!}(d_{k}) \big) - \rank_{\cK}
                                      \big( p_{!}(d_{k+1}) \big)\\
    & < r_{k} - \big( \rank_{\cK} (d_{k}) - (q+1-m) \big) -
      \big(\rank_{\cK} (d_{k+1}) -m \big) \\
    & = r_{k} - \rank_{\cK} (d_{k}) - \rank_{\cK} (d_{k+1}) + q+1 \\
    & = b_{k}^{\cK} (C) + q + 1\ ,
  \end{align*}
  whence $b_{k}^{\cK} \big( p_{!}(C) \big) \leq b_{k}^{\cK} (C) + q$
  so that \eqref{eq:inc_Betti} does not hold, as required.
\end{proof}

We can now apply Theorem~\ref{thm:rank_matrix} for each {\it
  fixed\/}~$j$ in the range $0 \leq j \leq q+1$ to the two
matrices~$D_{k}$ and~$D_{k+1}$, yielding a family of direct summands
of~$\bZ^{s*}$ characterising for which $p \in \hom(\bZ^{s}, \bZ^{t})$ the
inequalities~\eqref{eq:a} and~\eqref{eq:b} hold for our given choice
of~$j$. To characterise for which~$p$ inequality~\eqref{eq:inc_Betti}
holds, we allow $j$ to vary and take the union of all the groups~$G_j$
occurring. Collecting the information then results in the following:

\begin{theorem}
  \label{thm:jump_Betti}
  Let $C$ be a (not necessarily bounded) chain complex of finitely
  generated based free $R[\bZ^{s}]$\nbd-modules, with differentials
  given by matrices over~$R[\bZ^{s}]$. Let $q \geq 0$ and $k \in
  \bZ$.
  There are a number $\ell \geq 0$ and direct summands
  $G_{1},\, G_{2},\, \cdots,\, G_{\ell}$ of~$\bZ^{s*}$ such that for
  $p \in \hom(\bZ^{s}, \bZ^{t})$,
  \begin{displaymath}
    b_{k}^{\cK} \big( p_{!}(C) \big) > b_{k}^{\cK}(C) + q %
    \qquad \Longleftrightarrow \qquad %
    p \in \bigcup_{j=1}^{\ell} G_{j}^{t} \ . 
    \tag*{\qedsymbol}
  \end{displaymath}
\end{theorem}

\subsection*{Empty jump locus}

We finish the paper with a curious observation on the minors of
differentials in certain chain complexes of free modules. We keep the
notation from above: $R$ is a commutative ring, $\cK$ a hereditary set
of ideals, and $C$ a (possibly unbounded) chain complex consisting of
finitely generated free based $R[\bZ^{s}]$-modules. As before we
denote the $k$th differential by~$d_{k}$, and insist that $d_k$ is
given by multiplication by a matrix~$D_{k}$. In
Theorem~\ref{thm:jump_Betti} we consider $q=0$ and a fixed
$k \in \bZ$, and assume that $\ell = 0$ so the jump locus is the empty
set. This means that for all $p \in \hom(\bZ^{s}, \bZ^{t})$ the
induced complex $p_{!}(C)$ has the same $k$th \cK-\textsc{Betti}
number as~$C$, that is,
$b_{k}^{\cK} \big( p_{!} (C) \big) = b_{k}^{\cK} (C)$.

By definition of \textsc{Betti} numbers the equality
$b_{k}^{\cK} \big( p_{!} (C) \big) = b_{k}^{\cK} (C)$ necessitates
that $\rank_{\cK} (d_{i}) = \rank_{\cK} \big( p_{!}(d_{i}) \big)$ for
$i = k,k+1$. Write $r = \rank_{\cK} (d_{i})$. In case $r>0$ we let
$j \leq r$ be a positive integer. The set of $j$-minors of~$D_{i}$ is
not a \cK-set, by definition of the \cK-rank and
Lemma~\ref{lem:meaningful}. It follows that the set of $j$-minors of
$p_{*}(D_{i})$ is not a \cK-set either. (Indeed, if it was a \cK-set
then
$\rank_{\cK} \big( p_{!} (C) \big) < j \leq r = \rank_{\cK} (d_{i})$
so that $b_{i}^{\cK} \big( p_{!} (C) \big) > b_{i}^{\cK} (C)$ contrary
to our hypothesis.) In particular, not all $j$-minors of~$D_{i}$ can
be contained in the augmentation ideal of~$R[\bZ^{s}]$ as otherwise
the $j$-minors of $0_{*} (D_{i})$ would all be trivial, and would thus
form a \cK-set. We have shown:

\begin{proposition}
  Suppose that $b_{k}^{\cK} \big( p_{!} (C) \big) = b_{k}^{\cK} (C)$
  for all $p \in \hom(\bZ^{s}, \bZ^{t})$. Let $i = k,k+1$. For every
  positive integer $j \leq \rank_{\cK} (d_{i})$ at least one $j$-minor
  of~$D_{i}$ is not contained in the augmentation ideal
  of~$R[\bZ^{s}]$. \qed
\end{proposition}

The Proposition applies to the following special case: $R$ a unital
integral domain, $\cK = \cK_{0}$ and $C$ a contractible complex. For
then $p_{!}(C)$ is contractible as well since tensor products preserve
homotopies. It follows that $C$ and~$p_{!}(C)$ are acyclic, for
any~$p$; as $R$ is an integral domain, the (usual) \textsc{Betti}
numbers, corresponding to the specific choice of $\cK_{0}$ as
hereditary set of ideals, can be computed as the rank (in the usual
sense) of the homology modules, which are all trivial. That is, both
$C$ and~$p_{!}(C)$ have vanishing $k$th \textsc{Betti} number for all
$k \in \bZ$. As $\rank_{\cK_{0}} (d_{k}) > 0$ is equivalent to
$d_{k} \neq 0$, we conclude:

\begin{corollary}
  Suppose that $R$ is a unital integral domain, and that $C$ is a
  contractible chain complex of finitely generated free
  $R[\bZ^{s}]$-modules. For every non-zero differential~$d_{k}$
  of~$C$, and every positive $j \leq \rank(d_{k})$, at least one
  $j$-minor of its representing matrix~$D_{k}$ is not contained in the
  augmentation ideal of~$R[\bZ^{s}]$. \qed
\end{corollary}

\raggedright

\end{document}